\documentclass[10pt,twoside]{siamltex}
\usepackage{amsfonts,epsfig}
\usepackage{hyperref}

\usepackage{graphics}
\usepackage{epsfig}
\usepackage{amssymb,amsmath,mathrsfs}
\usepackage{amsfonts}
\usepackage{color}
\usepackage{mathtools}

\usepackage{tikz}
\usetikzlibrary{arrows}
\usetikzlibrary{decorations.markings}
\tikzstyle{block}=[draw opacity=0.7,line width=1.4cm]
\tikzset{
  big black arrow/.style={
    decoration={markings,mark=at position 1 with {\arrow[scale=2.5,black]{>}}},
    postaction={decorate},
    shorten >=0.4pt},
    line/.style={draw, ->}}

\newtheorem{remark}[theorem]{Remark}
\newtheorem{example}[theorem]{Example}

\numberwithin{equation}{section}

\begin{document}



\bibliographystyle{plain}
\title{
On Generalized Jacobi,  Gauss-Seidel and SOR Methods
}

\author{
Manideepa Saha\thanks{Department of Mathematics, National Institute of Technology Meghalaya, Shillong 793003, India
(manideepa.saha@nitm.ac.in).}
\and Jahnavi Chakrabarty\thanks{Department of Mathematics, National Institute of Technology Meghalaya, Shillong 793003, India.} 
}

\pagestyle{myheadings}
\markboth{M.\ Saha
and J.\ Chakrabarty
 }{On Generalized Jacobi,  Gauss-Seidel and SOR Methods }
\maketitle

\begin{abstract} In this paper generalization of Jacobi and Gauss-Seidel methods, introduced by Salkuyeh in 2007, is studied. In particular, convergence criteria for these methods are discussed. A generalization of successive overrelaxation~(SOR) method is proposed, and  its convergence properties for various classes of matrices are discussed. Advantages of generalized SOR method are established through  numerical experiments.


\end{abstract}



\section{Introduction} \label{intro3} 
The idea of solving large square systems of linear equations by iterative methods is certainly not new, dating back at least to Gauss [1823]. Jacobi, Gauss-Seidel and SOR methods are most stationary iterative methods that date to the late eighteenth century, but they find current application in problems where the matrix is sparse. 

Consider the linear system of equations
\begin{equation}\label{main_eqn1}
Ax=b
\end{equation}
where $A\in\mathbb{R}^{n,n}$,  and $b\in\mathbb{R}^{n}$. If $A$ is a nonsingular matrix with nonzero diagonal entries, then Jacobi, Gauss-Seidel, and SOR (Successive Over-relaxation) methods for solving $(\ref{main_eqn1})$ are given respectively as,
\begin{eqnarray}
 x^{(n+1)} &=& D^{-1}(E+F)x^{(n)}+D^{-1}b \nonumber\\
x^{(n+1)} &=&(D-E)^{-1}Fx^{(n)}+(D-E)^{-1}b \nonumber \\
x^{(n+1)} &=& (D-\omega E)^{-1}[\left(1-\omega)D+ \omega F\right]x^{(n)}+(D-\omega E)^{-1}b 
\end{eqnarray}
 where $D,-E,-F$ are the diagonal, strictly lower and upper triangular part of $A$, respectively.

Recently, in \cite{Sal07}, authors considered generalization of Jacobi and Gauss-Seidel methods by generalizing the diagonal matrix into a band matrix, and termed them as generalized Jacobi (GJ) and generalized Gauss-Seidel (GGS) methods, respectively. They proved that the convergence of GJ and GGS methods converge for strictly diagonally dominant(SDD) and for $M$-matrices.

It has been known that Jacobi and Gauss-Seidel method also converges for symmetric positive definite matrices(SPD), $L$-matrices and for $H$-matrices~\cite{Var00, BerP94, Saa03, You14}. In this paper we study the convergence of GJ ang GGS methods for the mentioned classes of matrices. We introduce a generalization of SOR method similar to that of GJ, or GGS method, and discuss the convergence of the proposed methods for above mentioned classes. 

The paper is organized as follows: In section~2, we study the convergence of GJ and GGS method for SPD, $L$-matrices and for $H$-matrices. In section ~3, generalized SOR methods is proposed and study the convergence of the method for SDD, SPD, $M$-, $L$-, and for $H$-matrices.  Lastly, in section~4, numerical examples are considered to illustrate the convergence of the proposed methods.
\section{Generalized Jacobi and Gauss-Seidel Method} In this section we describe  GJ and GGS iterative procedures, introduced in~\cite{Sal07}, and check the convergency of these methods for SPD-matrices, $L$-matrices and for $H$-matrices.

Let $A=(a_{ij})$ be an $n\times n$ matrix and $T_m=(t_{ij})$ be a banded matrix of bandwidth $2m+1$ defined as
\begin{center}
$t_{ij} = \left\{
        \begin{array}{ll}
            a_{ij}, & |i-j| \leq m \\
              0 ,& \text{ otherwise }
        \end{array}
    \right.$
    \end{center}
Consider the decomposition $A=T_m-E_m-F_m$, where $E_m$ and $F_m$ are the strict lower part and upper part of the matrix . In other words matrices are defined as following\\
\begin{eqnarray}\label{eqn2.2}
T_{m}=\left[\begin{array}{rrrr}
 a_{1,1} & \dots & a_{1,m+1} & 0 \\
\vdots  & \ddots & \ddots & \\
a_{m+1,1 } &  & & a_{n-m,n}\\
 & \ddots&\ddots & \\ 
0 & & a_{n,n-m} & a_{n,n}
\end{array}\right],  ~& E_{m}=\left[\begin{array}{rrr}
  0 & \dots & 0 \\
 -a_{m+2,1}   \\
\vdots &\ddots&\vdots  \\
-a_{n,1} & \dots & -a_{n-m-1,n}
\end{array}\right]\nonumber \\
\end{eqnarray}
\begin{center}
$F_{m}=\left[\begin{array}{rrrl}
0 & -a_{1,m+2}  &\ldots -a_{1,n} \\
\vdots &\ddots   &\vdots \\
0 & \ldots & -a_{n-m-1,n}\\ & &
\end{array}\right]$.\end{center}
%

Then GJ and GGS methods defined in\cite{Sal07}~for the system (\ref{main_eqn1}) are defined respectively as,
\begin{eqnarray}
 x^{(n+1)} &=& T_m^{-1}(E_m+F_m)x^{(n)}+T_m^{-1}b \label{eqn_GJ}\\
x^{(n+1)} &=&(T_m-E_m)^{-1}F_mx^{(n)}+(T_m-E_m)^{-1}b \label{eqn_GGS}
\end{eqnarray}
that is, corresponding to the splittings $A=M_m-N_m$, with  $M_m=T_m,~N=E_m+F_m$, and $M=T_m-E_m,~N=F_m$, respectively. Equivalently,  $H_{GJ}=T_m^{-1}(E_m+F_m)$, $H_{GGS}=(T_m-E_m)^{-1}F_m$ are the respective iterative matrices for GJ and GGS methods. Note that if $m=0$, then (\ref{eqn_GJ}) and (\ref{eqn_GGS}) will reduce to Jacobi and Gauss-Seidel methods, respectively.

We now define the classes of matrices, considered to study the convergence of the mentioned methods.

\begin{definition}\rm \cite{HorJ94, BerP94, Var00}
An $n \times n$ matrix $A=(a_{ij})$ is said to be strictly diagonally dominant (SDD) if 
\begin{equation}
|a_{ii}|> \sum\limits_{{j=1},j\neq 1}^n|a_{ij}|, i=1,2,3,.....,n
\end{equation}
\end{definition}
\begin{definition}\rm \cite{HorJ94, BerP94, Var00}
An $n \times n$ matrix $A=(a_{ij})$ is said to be symmetric positive definite (SPD) if $A$ is symmetric and $x^TAx>0$ for all $x\neq 0$.\\
\end{definition}

\begin{definition}\rm \cite{HorJ94, BerP94}
A matrix $A\in\mathbb{R}^{n,n}$ , is said to be an $M$-matrix if $A$ can be written as $A=sI-B$, where $B\geq 0$ (i.e., $B$ is entrywise nonnegative), and $s\geq\rho(B)$.
\end{definition}
\begin{definition}\rm \cite{HorJ94, BerP94}
A matrix $A\in \mathbb{C}^{n,n}$ is said to be an $H$-matrix if its comparison matrix $H(A)=(m_{ij})$ with  $m_{ii}=|a_{ii}|$ and $m_{ij}=-|a_{ij}|$, is an $M$-matrix.
\end{definition}
\begin{definition}\rm \cite{HorJ94, BerP94}
A matrix $A\in \mathbb{C}^{n,n}$ is said to be  an $L$-matrix if  for each $i,~a_{ii}>0$  and $a_{ij}\leq 0$,  for $i\neq j$.
\end{definition}

Jacobi and Gauss-Seidel method converge for SDD and $M$-matrices~\cite{HorJ94, BerP94, Saa03}. In~\cite{Sal07}, author proved that GJ and GGS methods converge for both these classes of matrices, as stated below:
\begin{theorem}\rm\cite{Sal07}
Let $A$ be an SDD-matrix. Then for any natural number $m\leq n$, GJ and GGS methods converge for any initial guess $x_0$. 
\end{theorem}

\begin{theorem}\rm\label{thm_GJ_GGS_M}\cite{Sal07}
If $A$ is an $M$-matrix, then for a given natural number $m\leq n$, both GJ and GGS methods are convergent for any initial guess $x_0$.
\end{theorem}

As it is known that both Jacobi and Gauss-Seidel methods also converge for SPD-matrices, $L$-matrices and for $H$-matrices~~\cite{BerP94, Saa03, You14, Var00}, we now discuss the convergence of GJ and GGS for these classes.

\begin{theorem}\label{thm1.2}\rm \cite{You14} If $A=M-N$ is a splitting of $A$, then the corresponding  iterative method $x^{(n+1)}=M^{-1}Nx^{(n)}+M^{-1}b$  for solving~(\ref{main_eqn1}) converges for any initial guess $x_{0}$ if and only if $\rho(M^{-1}N)<1$.
  \end{theorem}
  
Theorem~\ref{thm1.2}  is used to show that convergence of GJ and GGS methods.

Following theorem gives charaterization of $M$-matrices, which is used to prove the convergence of GJ and GGS method for $H$-matrices.
\begin{theorem}\rm\cite{BerP94}\label{charM} Let $A$ be a matrix with off-diagonal entries are non-positive and $A$ is nonsingular. Then following statements are equivalent:
\begin{itemize}
\item[(i)] $A$ is an $M$-matrix 
\item[(ii)] $A$ is semipositive, that is, there exists $x>0$ such that $Ax>0$.
\item[(iii)] $A$ is monotone, that is, $Ax\geq 0$ implies $x\geq 0$.
\item[(iv)] $A^{-1}\geq 0$.
\end{itemize}
\end{theorem}
\begin{theorem}\rm\cite{HorJ90}\label{thm_Horn_n} If $A\neq 0$, $x\geq 0$, and $x\neq 0$. If $Ax\geq\alpha x$, for some real $\alpha$, then $\rho(A)\geq\alpha.$

\end{theorem}

\begin{theorem}\label{thm_GJ_H}\rm If $A$ is an $H$-matrix, then GJ converges for any initial guess $x_{0}.$
\end{theorem}
\begin{proof}
Let $A$ be an $H$-matrix. Consider the decomposition of $A$ as $A=T_{m}-(E_{m}+F_{m})$,  for some $m$. Let $M=T_{m}=D+R_{m}$, and $N=E_{m}+F_{m}$ , where $D=\diag(A)$, and $R_{m}=T_{m}-D$. Let $H(A)$ be the comparison matrix of $A$, so that $H(A)$ is an $M$-matrix. Note that $H(A)=|D|-|R_{m}|-|E_m|-|F_m|$. Then $H(A)=M_{1}-N_{1}$ is the generalized Jacobi splitting of $H(A)$, where $M_{1}=|D|-|R_{m}|$  and $N_{1}=|E_m|+|F_m|$. Hence by Theorem~\ref{thm_GJ_GGS_M}, $\rho(M_{1}^{-1}N_{1})<1$.

Let $\lambda$ be any eigenvalue of $M^{-1}N$, and let $x\neq 0$ such that $M^{-1}Nx=\lambda x$, that is, $Nx=\lambda Mx$. Then $|\lambda|.|Mx|\leq |N|.|x|$ implies that 
\begin{equation}\label{eq1}
|\lambda|.|Dx+R_{m}x|\leq |N|.|x|
\end{equation}
Again,\\
 $|Dx+R_{m}x|=|Dx-(-R_{m}x)|\geq|\left(|Dx|-|R_{m}x|\right)|\geq |Dx|-|R_{m}x|=|D|.|x|-|R_{m}x|.$ Now equation~(\ref{eq1}) implies that
\begin{eqnarray}\label{eqn2}
&& |\lambda|\left(|D|.|x|-|R_{m}x|\right)\leq |N|.|x|\nonumber\\
&\Rightarrow& |\lambda|\left(|D|.|x|-|R_{m}|.|x|\right)\leq |N|.|x|\leq N_{1}|x|\ \text {    ~~~~~~~ as $|R_{m}x|\leq |R_{m}|.|x|$}\nonumber\\
&\Rightarrow& |\lambda| M_{1}|x|\leq N_{1}|x|
\end{eqnarray}
Since $M_1$ is an $Z$-matrix and $H(A)\leq M_{1}$, so $M_1$ is an invertible $M$-matrix and hence by Theorem~\ref{charM} $M^{-1}_{1}\geq 0$. Equation (\ref{eqn2}) implies that $|\lambda|.|x|\leq M_{1}^{-1}N_{1}|x|$. As $ M_{1}^{-1}N_{1}\geq 0$,  $|x|\geq 0$, and $x\neq 0$, so by Theorem~\ref{thm_Horn_n}$,~|\lambda|\leq\rho( M_{1}^{-1}N_{1})<1$. This shows that $\rho( M^{-1}N)<1$, and  GJ method converges.
%
\end{proof}
\begin{theorem}\label{thm_GGS_H}\rm If $A$ is an $H$-matrix, GGS method converges for any initial guess $x_{0}.$
\end{theorem}
\begin{proof}
Let $A$ be an $H$-matrix. Consider the generalized Gauss-Seidel splitting of $A=T_{m}-E_{m}-F_{m}$, for some $m$. Let $M=T_{m}-E_{m}=D+R_{m}-E_{m}$, and $N=E_{m}+F_{m}$ , where $D=\diag(A)$, and $R_{m}=T_{m}-D$. Let $H(A)$ be the comparison matrix of $A$, so that $H(A)$ is an $M$-matrix. Note that $H(A)=|D|-|R_{m}|-|E_m|-|F_m|$. Then $H(A)=M_{1}-N_{1}$ is the generalized Gauss-Seidel splitting of $H(A)$, where $M_{1}=|D|-|R_{m}|-|E_{m}|$  and $N_{1}=|F_m|$. As we know that GGS method converges for $M$-matrices, so $\rho(M_{1}^{-1}N_{1})<1$.

Let $\lambda$ be any eigenvalue of $M^{-1}N$, and let $x\neq 0$ such that $M^{-1}Nx=\lambda x$, that is, $Nx=\lambda Mx$. Then $|\lambda|.|Mx|\leq |N|.|x|$ implies that 
\begin{equation}\label{eq1.2}
|\lambda|.|Dx+R_{m}x-E_{m}x|\leq |F_{m}|.|x|=N_{1}.|x|
\end{equation}
Again \[\small |Dx+R_{m}x-E_{m}x|\geq|\left(|Dx|-|R_{m}x-E_{m}x|\right)|\geq |D|.|x|-|R_{m}x-E_{m}x|\geq |D|.|x|-|R_{m}x|-|E_{m}x|\]
that is, $|Dx+R_{m}x-E_{m}x|\geq |D|.|x|-|R_{m}|.|x|-|E_{m}|.|x|=M_{1}|x|$. So equation~(\ref{eq1.2}) implies that
\begin{equation}\label{eqn2}
|\lambda| M_{1}|x|\leq N_{1}|x|
\end{equation}
Since $M_1$ is an $Z$-matrix and $H(A)\leq M_{1}$, so $M_1$ is an invertible $M$-matrix and hence $M^{-1}_{1}\geq 0$. Equation (\ref{eqn2}) implies that $|\lambda|.|x|\leq M_{1}^{-1}N_{1}|x|$. As $ M_{1}^{-1}N_{1}\geq 0$,  $|x|\geq 0$, and $x\neq 0$, so $|\lambda|\leq\rho( M_{1}^{-1}N_{1})<1$. This shows that $\rho( M^{-1}N)<1$, and  GGS method converges.
\end{proof}

Following examples show that neither GJ nor GGS may converge for SPD and for $L$-matrices.

\begin{example}\label{exam2}\rm Consider the symmetric positive definite matrix\[\begin{array}{rr}
A=\left[\begin{array}{rrrr}
410 & -195 & -90\\
-195 & 151 & 112\\
-90 & 112 & 132
\end{array}\right]
\end{array}\]
Take $m=1$ so that $A=T_1-E_1-F_1$, where
\[T_{1}=\left[\begin{array}{crr}
410 & -195 & 0\\
-195 & 151 & 112\\
0 & 112 & 132
\end{array}\right], ~ E_1=\left[\begin{array}{rrr}
0 & 0 & 0\\
0 & 0 & 0\\
-90 & 0 & 0
\end{array}\right], \text{ and }F_1=E_1^T\]

If $M_{1}=T_{1}$, and $N_{1}=E_{1}+F_{1}$, then $\rho(H_{GJ})=\rho(M_{1}^{-1}N_{1})=1.5883>1$. Also if we take  $M_{1}=T_{1}-E_{1}$, and $N_{1}=F_{1}$, then $\rho(H_{GGS})=\rho(M_{1}^{-1}N_{1})=30.1584>1$. This shows that neither GJ nor GGS converge. 

\end{example}
\begin{example}\label{exam3}\rm  Let us consider the $L$-matrix 
\[A=\left[\begin{array}{rrr}
1 & -1 & -5 \\
-2 & 3 & -4 \\
-1 & -5 & 3
\end{array}\right].\]
If $m=1$,  $\rho(H_{GJ})=\rho\left(T_{1}^{-1}(E_1+E_{1}^{T})\right)=2.8689$ and  $\rho(H_{GGS})=\rho\left((T_1-E_1)^{-1}E_{1}^{T}\right)=3.0952$,  so both GJ and GGS do not converge. Note that the matrix $A$ is an $L$-matrix, but not an $M$-matrix.
\end{example}
\section{Generalized SOR method} Successive Over-relaxation (SOR) method is a variant of Gauss-Seidel method, which can be used to accelerate the convergence of Gauss-Seidel method. If $D,~-E$ and $-F$ are respectively, diagonal, strictly lower triangular, and strictly upper triangular parts of a matrix $A$, SOR iterative method \cite{Had00} for solving the linear system (\ref{main_eqn1}) is given by,
 \begin{eqnarray}\label{eqn3.1}
x^{(n+1)} &=& (D-\omega E)^{-1}[\left(1-\omega)D+ \omega F\right]x^{(n)}+(D-\omega E)^{-1}b 
 \end{eqnarray} 
In equation (\ref{eqn3.1}), $\omega$ is called as relaxation factor. Note that SOR  method (\ref{eqn3.1}) is corresponding to the splitting $\omega A=M-N=(D-\omega E)-((1-\omega)D+\omega F)$ of $\omega A$. For $\omega=1$, SOR method (\ref{eqn3.1}) coincide with Gauss-Seidel method, which therefore also called as relaxation.  For $0<\omega<1$, the precise name of (\ref{eqn3.1}) is underrelaxation method, whereas the term overrelaxation suits $\omega>1$. But we use the term SOR for (\ref{eqn3.1}). It is well known that if $0<\omega<2$, then SOR method (\ref{eqn3.1}) converges.
 
We now introduce generalized SOR method, called as GSOR method, similar to that of GJ and GGS .
\begin{definition}\rm
 For $1\leq m<n$, let $T_{m}$, $E_{m}$ and $F_{m}$ be the matrix defined in (\ref{eqn2.2}) such that $A=T_{m}-E_{m}-F_{m}$. We define GSOR method for the linear system $Ax=b$ by,
 \begin{eqnarray}\label{GSOR}
 x^{(n+1)}=& (T_m-\omega E_m)^{-1}\left[(1-\omega)T_m+\omega F_m\right]x^{(n)}+\omega(T_m-\omega E_m)^{-1}b
 \end{eqnarray}
 \end{definition}
Then $H_{\rm GSOR}=(T_m-\omega E_m)^{-1}[(1-\omega)T_m+\omega F_m]$ is the iterative matrix for the GSOR method~(\ref{GSOR}). Note that the method (\ref{GSOR}) is corresponding to splitting $\omega A=(T_m-\omega E_m)-\left((1-\omega)T_m+\omega F_m\right)$ of $\omega A$. We now  discuss convergence analysis of GSOR method for various class of matrices.

Following theorem gives a necessary condition for convergence of SOR method.
\begin{theorem}\label{nec_SOR}\rm \cite{Had00,BerP94} A necessary condition for SOR method to converge is $|\omega-1|<1.$\end{theorem}

We now state the classes of matrices for which  SOR method converges, and check the convergence of GSOR method for these classes.
\begin{theorem}\label{thm_Hac94}\rm\cite{Had00}
Let $A\in\mathbb{R}^{n\times n}$ be a symmetric positive definite matrix. If $\omega\in(0,2)$, SOR method converges for all initial iterates $x_{0}\in \mathbb{R}^n$. 
\end{theorem}

\begin{remark}\rm Theorem~\ref{thm_Hac94} and Theorem~\ref{nec_SOR} show that for SPD-matrices, SOR method converges if and only if $0<\omega<2$. But the result doesn't carry over to GSOR method. Following examples illustrate the fact.
\end{remark}
\begin{example}\rm Consider the symmetric positive definite matrix 
\[A=\left[\begin{array}{rrrr}
5 & 1 & 4 & 2\\
1 & 5 & 3 & 2\\
4 & 3 & 5 & 4\\
2 & 2 & 4 & 5
\end{array}\right].\]
Take $m=2$ so that $M_2=(T_2-\omega E_2)$, and $N_2=(1-\omega)T_2+\omega F_2$. For $\omega =1.8$, $\rho(M_{2}^{-1}N_{2})=1.1511$. Hence, GSOR doesn't converge. So, Theorem~\ref{thm_Hac94} does not hold for GSOR method.
\end{example}
\begin{example}\rm Consider the symmetric positive matrix $A$ defined in example~\ref{exam2}.  Take $m=1$, and $\omega =0.6$. If $M_{1}=T_{1}-\omega E_{1}$ and $N_{1}=(1-\omega)T_{1}+\omega F_{1}$, then $\rho(H_{GSOR})=\rho(M_{1}^{-1}N_{1})=1.7649>1$. Hence GSOR method doesn't for converge for symmetric positive definite matrices even if $0<\omega<1$.
\end{example}

\begin{theorem}\label{thm_Had}\rm \cite{Had00, BouV79}
 If $A \in \mathbb{R}^{n}$ is an $L$-matrix and $\omega \in (0,1]$ then
 $\rho {(H_J)}< 1 $ if and only if  $\rho({H_{SOR}})<1$. 
\end{theorem}

Since GGS method may not converge for $L$-matrices, GSOR method also may not converge for $\omega=1$. Following example shows that even if $0<\omega<1$, GSOR method (\ref{GSOR}) may not converge for $L$-matrices.

\begin{example}\label{exp3.1}\rm Consider the $L$-matrix $A$ defined in example~\ref{exam3}.  Take $m=1$, and $\omega =0.9$. If $M_{1}=T_{1}-\omega E_{1}$ and $N_{1}=(1-\omega)T_{1}+\omega F_{1}$, then $\rho(H_{GSOR})=\rho(M_{1}^{-1}N_{1})=2.6705>1$. Hence, GSOR method doesn't for converge for $L$-matrices.
Again if we take $\omega =0.4$, then $\rho(H_{GSOR})=0.6<1$. In this case GSOR method converges. Note that $0<\omega <1$ for both the examples.
\end{example}
\begin{remark}\rm From example~\ref{exp3.1} and example~\ref{exam3}, we observe that Theorem~\ref{thm_Had} does not hold for GJ and GSOR methods.
\end{remark}

Next important class of our consideration is the class of $M$-matrices. From the following theorem, it is known that if $0<\omega\leq 1$, then SOR method converges for $M$-matrices.

\begin{theorem}\label{thm_SOR_M}\rm~\cite{BerP94} Let $A$ be matrix with off-diagonal entries are non-positive, and $0<\omega\leq 1$. Then SOR method converges if and only if $A$ is a nonsingular $M$-matrix.
\end{theorem}

We now prove the convergence of GSOR method for nonsingular $M$-matrices, for which we require the concept of regular splitting of a matrix and few characterizations of $M$-matrices.

\begin{definition}\rm For $n\times n$ matrices $A, ~M, ~N$, a splitting $A=M-N$ is called a regular splitting of $A$ if  $M^{-1}\geq 0$ and $N\geq 0$.
\end{definition}
\begin{theorem}\rm\cite{Saa03}\label{thm_reg} Let $A=M-N$ be a regular splitting of $A$. Then $\rho(M^{-1}N)<1$ if and only if $A$ is nonsingular and $A^{-1}\geq 0$.
\end{theorem}


\begin{theorem}\label{Thm_GSOR_M}\rm
If $0<\omega\leq 1$, then GSOR method converges for $M$-matrices.
\end{theorem}

\begin{proof} Consider the decomposition $A=T_{m}-E_{m}-F_{m}$ of $A$. Let $M=T_{m}-\omega E_{m}$, and $N=(1-\omega)T_{m}+\omega F_{m}$. Then $\omega A=M-N$ is the GSOR splitting.  To show that $\rho(M^{-1}N)<1$. 

Since $A$ is an $M$-matrix, $T_{m}$ is also an $M$-matrix, and hence $T_{m}^{-1}\geq 0$. If $M_{1}=I-\omega T_{m}^{-1}E_{m}$, and $N=(1-\omega)I+\omega T_{m}^{-1}F_{m}$, then $M^{-1}N=M_{1}^{-1}N_{1}$ so that we need to show that $\rho(M_{1}^{-1}N_{1})<1$.
We now show  that $M_{1}$ is an $M$-matrix. Note that $M_{1}$ is an $Z$-matrix. Since $A$ is an $M$-matrix, by Theorem~\ref{charM} there is $x>0$ such that $Ax>0$.\\
Now,
\begin{eqnarray}
Ax>0&\Rightarrow& (T_{m}-E_{m}-F_{m})x>0\nonumber\\
&\Rightarrow& (T_{m}-E_{m})x>0~~~~~~~~~~\text{ as } F_{m}\geq 0,~x>0\nonumber\\
&\Rightarrow& (T_{m}-\omega E_{m})x>0~~~~~~~~~~\text{ as } 0<\omega\leq 1\nonumber\\
&\Rightarrow& T_{m}^{-1}(T_{m}-\omega E_{m})x>0 ~~~~~~~~~~\text{ as } T_{m}^{-1}\geq 0\nonumber\\
&\Rightarrow& M_{1}x>0\nonumber
\end{eqnarray}
Hence there is $x>0$ and $M_{1}x>0$, which implies $M_{1}^{-1}\geq 0$ by Theorem~\ref{charM}. Also $N_{1}=(1-\omega)I+\omega T_{m}^{-1}F_{m}\geq 0$, so $\omega T_{m}^{-1}A=M_{1}-N_{1}$ is a regular splitting of $\omega T_{m}^{-1}A$. But $\omega T_{m}^{-1}A$ is nonsingular and $(\omega T_{m}^{-1}A)^{-1}=\omega^{-1}A^{-1}T_{m}\geq 0$ since both $A^{-1}$ and $T_{m}$ are nonnegative. Hence $\rho(M_{1}^{-1}N_{1})<1$ by Theorem~\ref{thm_reg}. Thus, GSOR method converges.
\end{proof}
\begin{theorem}\rm\label{thm_SOR_M2} Let $A$ be an $M$-matrix, and $0<\omega<\frac{2}{1+\rho(H_{GJ})}$. If $\rho(T_{m}^{-1}E_{m})<\frac{1}{\omega}$, then GSOR method converges for any initial guess $x_{0}$.
\end{theorem}
\begin{proof} If $0<\omega\leq 1$, then it is true by Theorem~\ref{Thm_GSOR_M}. Let $\omega >1$. Consider the matrix $G=(I-\omega T_{m}^{-1}E_{m})^{-1}((\omega-1)T_{m}+\omega E_{m})$. Since $\rho(T_{m}^{-1}E_{m})<\frac{1}{\omega}$ and $T_{m}^{-1}E_{m}\geq 0$,  $(I-\omega T_{m}^{-1}E_{m})^{-1}\geq 0$ and so $G\geq 0$. Note that $|H_{GSOR}|\leq G$, which implies that $\rho(H_{GSOR})\leq \rho(G)$. It suffices to show that $\rho(G)<1$.\\
Let $\lambda=\rho(G)$ and $x\geq 0$ be the Perron vector of $G$ so that $Gx=\lambda x$, which implies 
\begin{equation}\label{eqn_M2}
(\omega T_{m}^{-1}F_{m}+\lambda\omega T_{m}^{-1}E_{m})=(\lambda +1-\omega)x
\end{equation}
 If $\lambda\geq 1$, then from equation (\ref{eqn_M2}), we have that $(1+\lambda-\omega)x\leq(\lambda\omega H_{GJ})x$, that is, $\lambda(1-\omega\rho(H_{GJ}))\leq \omega -1$ which implies that, $\omega\geq\frac{2}{1+\rho(H_{GJ})}$, since $\lambda \geq 1$. This is a contradiction. Hence $GSOR$ method converges.
\end{proof}

\begin{remark}\rm Theorem~\ref{thm_SOR_M2} shows that if $\rho(B_{GJ})$ is  small, then GSOR method may converge for $1<\omega<2$.
\end{remark}

Following theorem assures the convergence of SOR method for SDD matrices whenever $0<\omega\leq 1$.
\begin{theorem}\label{thm_DD_SOR}\rm\cite{BerP94} If $A\in \mathbb{C}^{n,n}$ is a SDD-matrix, then Jacobi method converge for $A$. Also if $0<\omega\leq 1$, then SOR method also converges.
\end{theorem}

In next theorem  we extend the above theorem for GSOR method.

\begin{theorem}\label{thm_DD_GSOR}\rm
If $0<\omega\leq1$, GSOR method converges for SDD-matrices.
\end{theorem}

\begin{proof}\rm
Let $A$ be a SDD-matrix. Consider the decomposition $A=T_{m}-E_{m}-F_{m}$ of $A$. Let $M=T_{m}-\omega E_{m}$, and $N=(1-\omega)T_{m}+\omega F_{m}$.  To show that $\rho(M^{-1}N)<1$. 

 Take $H=M^{-1}N$, and $\lambda \in \sigma(H)$. If $|\lambda| \geq 1$, then

\(\begin{array}{lll}
 \det (\lambda I-H)=0 &\Rightarrow& \det(\lambda I-M^{-1}N)=0\nonumber\\
 &\Rightarrow& \det[ M^{-1}(\lambda M-N)]= 0~~~~~~~~~~  \nonumber\\
 &\Rightarrow& \det (\lambda M-N)=0~~~~~~~~~~  \nonumber\\
 &\Rightarrow& \det\left[\lambda T_m - \omega \lambda E_m -(1-\omega)T_m-\omega F_m \right]=0 ~~~~~~~~~~\nonumber\\
 &\Rightarrow& \det \left[(\lambda-1 + \omega) T_m -\omega \lambda E_m -\omega F_m \right]=0 \nonumber\\
 &\Rightarrow& \det \left[ {T_m -\frac{\lambda \omega }{\lambda-1 + \omega}E_m -\frac{\omega}{\lambda-1 + \omega}F_m} \right]=0 \nonumber
\end{array} \)

Next, we show that
$\dfrac{|\lambda \omega|}{|\lambda -1 +\omega|} < 1$ and $\dfrac{\omega}{|\lambda -1 +\omega|} < 1$.
Write $\lambda=re^{i\theta}$. $|\lambda|>1,~ r \geq 1$. Then

\(\begin{array}{lll}
|\lambda-1+\omega|^2-|\lambda \omega|^2 &=[~\lambda -(1-\omega)]~[\bar{\lambda}-(1-\omega)]-r^2\omega ^2 \nonumber \\
&= r^2-[~\lambda (1-\omega)+\bar{\lambda}(1-\omega)]+(1-\omega)^2-r^2\omega^2 \nonumber\\
&= r^2-2r(1-\omega)\cos\theta  +(1-\omega)^2-r^2\omega^2 \nonumber\\
&\geq r^2-2 r (1-\omega) +(1-\omega)^2-r^2\omega^2 \nonumber \\
&=( r-1+\omega)^2 - r^2\omega^2 \nonumber\\
&=( r-1 )(1-\omega)~~[~r(1+\omega)-(1-\omega)] \nonumber\\
&\geq 0 \nonumber ~~~~~~~~~~~~~[\text { as } r>1,~\omega <1 , ~r(1+\omega)>1 \text{ and } (1-\omega)<1]
\end{array}\)\\\\
Thus,  $|\lambda -1 +\omega| \geq |\lambda \omega| \geq \omega$ implies that
$\dfrac{|\lambda \omega|}{|\lambda -1+\omega|}\leq 1$ and
$\dfrac{ \omega}{|\lambda -1+\omega|}\leq 1$.

Let $S= T_m -\alpha E_m -\beta F_m $, where  $\alpha =\dfrac{\lambda \omega}{\lambda -1 +\omega}$ and $\beta= \dfrac{\omega}{\lambda-1+\omega}$, so that $|\alpha| \leq 1$, $|\beta | \leq 1$.
We now show that $S$ is a SDD-matrix. Note that 
\begin{center}
$S_{ij} = \left\{
        \begin{array}{rl}
               a_{ij}, & |i-j| \leq m \\
              -\alpha a_{ij}  ,& i>j+m\\
              -\beta a_{ij}  ,& j>i+m 
        \end{array}
    \right.$
    \end{center}
    
For any $i$, we have

\begin{center}  \(\begin{array}{ll}
   \sum\limits_{i\neq j}|S_{ij}| 
   &= \sum\limits_{|i-j| \leq m}|a_{ij}| + |\alpha| \sum\limits_{i>j+m}|a_{ij}| + |\beta| \sum\limits_{j>i+m}|a_{ij}| \nonumber\\
   &\leq  \sum\limits_{|i-j| \leq m}|a_{ij}| + |\alpha| \sum\limits_{i>j+m}|a_{ij}| + |\beta| \sum\limits_{j>i+m}|a_{ij}| \nonumber\\
   &= \sum\limits_{i\neq j}|a_{ij}| \nonumber\\ 
   &< |a_{ii}| =|S_{ii}| \nonumber
   \end{array}\)\end{center}
   
This shows that $S$ is SDD-matrix and hence invertible, which is a contradiction to the fact that $\det S=0$ . So $|\lambda| < 1$. Thus GSOR method converges for SDD-matrices.
\end{proof}

In Theorem~\ref{thm_GJ_H} and~\ref{thm_GGS_H}, we proved that GS and GGS methods converge for the class of $H$-matrices.  We now prove it for GSOR method.

\begin{theorem}\rm If $A$ is an $H$-matrix, and $0<\omega\leq 1$, then GSOR method converges.
\end{theorem}
\begin{proof}
Let $A$ be an $H$-matrix and decompose $A$ as $ A= T_{m}-E_{m}-F_{m}$. If $D=\diag(A)$, we write $M=T_{m}-\omega E_{m}=D+R_{m}-\omega E_{m}$ and $N=(1-\omega)T_{m}+\omega F_{m}$, where $T_{m}=R_{m}+D$. Let $B$ be the comparison matrix of $A$, so that $B$ is an $M$-matrix. Note that $B=|D|-|R_{m}|-|E_m|-|F_m|$. If $M_{1}=|D|-|R_{m}|-\omega |E_{m}|$  and $N_{1}=(1-\omega)(|D|-|R_{m}|)+ \omega |F_{m}|$, then $\omega B=M_{1}-N_{1}$ is the GSOR splitting. By Theorem~\ref{thm_SOR_M}, we have that $\rho(M_{1}^{-1}N_{1})<1$.

Suppose that $\lambda \in \sigma({M^{-1}N})$ and $|\lambda| \geq 1$. Then, 
\begin{eqnarray}\label{eqn3.3}\rm
 |det (\lambda I-M^{-1}N)|=0 
&\Rightarrow& \rm det (\lambda M -N)= 0 \nonumber\\
&\Rightarrow& \rm det \left[\lambda T_m - \lambda \omega E_m -(1-\omega)T_m -\omega F_m \right]= 0 \nonumber\\
&\Rightarrow& \rm det \left[T_m - \frac{\lambda \omega}{\lambda -1+\omega}E_{m} - \frac{\omega}{\lambda -1 +\omega} F_{m} \right] = 0 
\end{eqnarray}

Set $a=\dfrac{\lambda \omega}{\lambda-1+\omega}, ~b=\dfrac{\omega}{\lambda -1 +\omega}$. Since $0 < \omega < 1$, as in Theorem~\ref{thm_DD_GSOR}, it can be verified that  $|a| \leq 1\text{  and  }|b| \leq 1$. From equation~(\ref{eqn3.3}), there exists $x (\neq 0)\in \mathbb{R}^n$ such that
 \begin{align}\label{eqn3.4}
 T_m x - a E_m x - b F_m x = 0
 &\Rightarrow |T_m x| =  |a E_m x + b F_m x| \nonumber\\
 &\Rightarrow |T_m x| \leq |a|.|E_m x| + |b|.|F_m x| \nonumber\\
 &\Rightarrow |T_m x| \leq |E_m|.|x| + |F_m|.|x| 
 \end{align}
Now,
\begin{align}\label{eqn3.5}
|T_m x| = |D x +R_m x| \geq \mid|Dx| - |R_m x| \mid \nonumber\\
\geq |D|.|x| - |R_m|.|x| 
\end{align}
From equations (\ref{eqn3.4}) and (\ref{eqn3.5}) we have that $|D|.|x| - |R_m|.|x|\leq |E_m|.|x| + |F_m|.|x|$, that is $B|x|\leq 0$, or equivalently, $B(-|x|)\geq 0$. By the monotone property of $M$-matrices, $-|x|\geq 0$, that is, $|x|=0$, which is a contradiction. Hence, $|\lambda| < 1$, that is, $\rho (M^{-1}N) < 1 $. Thus GSOR method converges for $H$-matrices.
\end{proof}

\section{Numerical Illustration} In this section numerical examples are considered to illustrate the convergence of methods discussed in this paper. Consider the equation
\begin{eqnarray}\label{eqn_lap}-\Delta u+g(x,y)u=&f(x,y),&~~~~(x,y)\in\Omega= [0,1]\times [0,1]\nonumber\\
=&0,&\text{ ~~~~on }\partial\Omega
\end{eqnarray}
On uniform mesh $p_{ij}=(ih,jh),~h=\frac{1}{n+1}, ~i,j=1,2,\ldots n+1$, if $u_{ij}=u(p_{ij}),~g_{ij}=g(p_{ij}),~f_{ij}=f(p_{ij})$, discretizing Laplacian by the central difference formula at the mesh-point $p_{ij}$, equation~(\ref{eqn_lap}) yields
\begin{equation}\label{eqn_disc}
-u_{i-1,j}-u_{i,j-1}+(4+h^2g_{ij})u_{ij}-u_{i,j+1}-u_{i+1,j}=h^{2}f_{ij}
\end{equation}

The equations in (\ref{eqn_disc}) can be written as the system of linear equations $Ax=b$, where $x$ and $b$ are vectors of length $n^2$ and of the form
\begin{eqnarray}
x=&[u_{11},\ldots,u_{1n};u_{21},\ldots,u_{2n};\ldots;u_{n1},\ldots,u_{nn}]^{T}\nonumber\\
b=& h^2[f_{11},\ldots,f_{1n};f_{21},\ldots,f_{2n};\ldots;f_{n1},\ldots,f_{nn}]^{T}
\end{eqnarray}
The co-efficient matrix $A$ is an $n^2\times n^2$ block band matrix with block-bandwith $3$ in the form
\[A=\left[\begin{array}{rrrrrr}
G_{11} & -I & 0 & \ldots & 0 & 0\\
-I & G_{22} & -I &\ldots & 0 & 0\\
\vdots &\vdots &\vdots&\ddots&\vdots\\
0 &0 & 0& \ldots & G_{n-1,n-1} & -I\\
0 &0 & 0& \ldots & -I &G_{n,n}
\end{array}\right], ~~G_{ii}=\left[\begin{array}{rrrrr}
\alpha_{i1} & -1 & 0 &  \ldots & 0\\
-1 & \alpha_{i2} & -1 &\ldots & 0\\
\vdots &\vdots &\vdots&\ddots&\vdots\\
0 &0 & 0& \ldots &\alpha_{in}
\end{array}\right]\]
where $\alpha_{ij}=4+h^2g_{ij}$. 

Numerical experiment is done by taking $g(x,y)$ as $x+y,~0, \exp(xy)$ and $-\exp(4xy)$ in (\ref{eqn_lap}) and the function $f(x,y)$ is choosen so that the exact solution of the system is $x=[1,1,\ldots,1]^{T}$. The initial approximation $x_{0}$ of the solution is chosen to be the zero-vector,  maximum number of iteration considered is $10000$, and the stoppong criteria is chosen as $\|x^{(n+1)}-x^{(n)}\|_{2}\leq 10^{-7}$. Results obtained for the different functions $g$ using different methods considered in the previous sections are listed below in terms of number of iterations and timing in seconds, for $n=20, 30,40$. For GJ, GGS and GSOR methods, we take $m=1$, and in addition to that we let $\omega=1.5$ for both SOR and GSOR methods. Note that the matrix $A$ is an M-matrix.\\

\begin{table}[htbp]
\caption{\rm Numerical result for $g(x,y)=x+y$}
\centering
\begin{tabular}{||p{1.2cm}|p{2cm}|p{2cm}|p{2cm}|p{2cm}||}
\hline
n & GJ & GGS & SOR & GSOR \\ \hline
20 & 619(4.85)& 322(2.51) &211(1.66) &105(0.84)\\ \hline
30 &1336(69.35) & 695(35.39) & 466(23.74)&240(12.55)\\ \hline
40 & 2312(576.52)& 1204(310.26) &815(201.33) &422(87.77)\\ \hline
\end{tabular}
\vspace{5mm}

\caption{\rm Numerical result for $g(x,y)=0$}
\centering
\begin{tabular}{||p{1.2cm}|p{2cm}|p{2cm}|p{2cm}|p{2cm}||}
\hline
n & GJ & GGS & SOR & GSOR \\ \hline
20 & 652(5.23) & 339(3.75) & 222(2.56) & 112(1.99)\\ \hline
30 &1405(71.43) & 731(37.20) & 491(25.02) & 253(13.86)\\ \hline
40 & 2429(556.29) & 1264(267.28) & 856(180.04) & 444(93.01)\\ \hline
\end{tabular}
\end{table}

\begin{table}[htbp]
\caption{\rm Numerical result for $g(x,y)=\exp(xy)$}
\centering
\begin{tabular}{||p{1.2cm}|p{2cm}|p{2cm}|p{2cm}|p{2cm}||}
\hline
n & GJ & GGS & SOR & GSOR \\ \hline
20 & 611(4.9) & 318(2.56) & 208(1.67) & 104(0.95)\\ \hline
30 & 1319(71.11) &687(37.08) & 460(24.85) & 237(12.80)\\ \hline
40 & 2282(498.34)& 1188(249.16) &804(169.05) & 417(94.33)\\ \hline
\end{tabular}
\end{table}
\begin{table}[htbp]
\caption{\rm Numerical result for $g(x,y)=-\exp(4xy)$}
\centering
\begin{tabular}{||p{1.2cm}|p{2cm}|p{2cm}|p{2cm}|p{2cm}||}
\hline
n & GJ & GGS & SOR & GSOR \\ \hline
20 & 824(9.54) &427(4.93) & 282(3.26) &143(1.70)\\ \hline
30 & 1736(88.39)& 899(55.33)&606(43.24) & 313(17.60)\\ \hline
40 & 2972(748.34)&1540(350.88) &1045(220.80) &543(114.47)\\ \hline
\end{tabular}
\end{table}

\newpage

\section{Conclusion} In this paper we considered generalization of Jacobi, Gauss-Seidel and SOR iterative methods for solving system of linear equations. In~\cite{Sal07}, author provided a generalization of Jacobi (GJ) and Gauss-Seidel(GGS)~methods and proved that these two methods converge for strictly diagonally dominant matrices and for $M$-matrices. In this paper, we have proved  the convergence of GJ and GGS methods for $H$-matrices. Examples have been considered to illustrate the fact that these methods may not converge for symmetric positive definite matrices and for $L$-matrices.

We have proposed GSOR method, a generalization of SOR method, by generalizing diagonal matrices to banded matrices. We have proved that GSOR method converges for strictly diagonally dominant matrices, $M$-matrices, and for $H$-matrices. Also, illustrated the fact that  GSOR method may not converge for symmetric positive definite matrices and for $L$-matrices.  Numerical experiment shows  that GSOR method is more effective than conventional SOR method, GJ and GGS method.


\end{document}